\documentclass{amsart}

\usepackage{amssymb}
\usepackage{amsmath,amsfonts,amsthm,amstext}
\usepackage{amsmath}
\usepackage{mathrsfs}  
\usepackage[all]{xy}
\usepackage{xcolor}

\newtheorem{theorem}{Theorem}[section]
\newtheorem{corollary}[theorem]{Corollary}
\newtheorem{lemma}[theorem]{Lemma}
\newtheorem{prop}[theorem]{Proposition}
\newtheorem{cor}[theorem]{Corollary}

\newtheorem*{inote}{Remark}
\def\mapnew#1{\smash{\mathop{\longrightarrow}\limits^{#1}}}

\numberwithin{equation}{section}
\def\-{\overline}
\newcommand{\R}{\mathbb{R}}
\newcommand{\Z}{\mathbb{Z}}

\newcommand{\BZ}{\mathbb{Z}}

\theoremstyle{definition}

\begin{document}
\title[On the $\Sigma$-invariants of the Bestvina-Brady groups ]{
On the Bieri-Neumann-Strebel-Renz $\Sigma$-invariants of the Bestvina-Brady groups} 
\author{Dessislava H. Kochloukova, Luis Mendon\c{c}a}
\address{
\newline State University of Campinas (UNICAMP), SP, Brazil \newline
  Heinrich-Heine-Universit\"at D\"usseldorf (HHU), Germany
\newline
email : desi@unicamp.br, mendonca@hhu.de} 
\email{}
\date{}

\thanks{The first named author was   partially supported by CNPq grant 301779/2017-1  and by FAPESP grant 2018/23690-6.}

\begin{abstract}  We study the Bieri-Neumann-Strebel-Renz invariants  and  we prove the following criterion: for groups $H$ and $K$ of type $FP_n$ such that $[H,H] \subseteq K \subseteq H$ and a character $\chi : K \to \R$ with $\chi([H,H]) = 0$ we have $[\chi] \in \Sigma^n(K, \Z)$ if and only if $[\mu] \in \Sigma^n(H, \Z)$ for every character $\mu : H \to \R$ that extends $\chi$. The same holds for the homotopical invariants $\Sigma^n(-)$ when $K$ and $H$ are groups of type $F_n$. We use these criteria to complete the description of the $\Sigma$-invariants of  the Bieri-Stallings groups $G_m$ and more generally to describe the $\Sigma$-invariants of the Bestvina-Brady groups. We also show that the ``only if'' direction of such criterion holds if we assume only that $K$ is a subnormal subgroup of $H$, where both groups are of type $FP_n$. We apply this last result to wreath products.
\end{abstract}

\maketitle

\section{Introduction} 

In \cite{Wall} Wall defined a group $G$ to be of  homotopical type $F_n$  if there is a classifying space $K(G,1)$ with finite $n$-skeleton. The homotopical type $F_2$ coincides with  finite presentability (in terms of generators and relations). A homological version of the homotopical type $F_n$ was defined by Bieri in \cite{Bieribook}:
a group $G$ is of homological type $FP_n$ if the trivial $\mathbb{Z} G$-module $\mathbb{Z}$ has a projective resolution with all modules finitely generated in dimension $\leq n$. Though every group of type $F_n$ is of type $FP_n$ the converse does not hold for $n \geq 2$. There are groups constructed by Bestvina and Brady in \cite{B-B} that are subgroups of right-angled Artin groups and are of type $FP_{\infty}$ but are not finitely presented.

The first $\Sigma$-invariant was defined by Bieri and Strebel in \cite{B-S}, where it was used to classify all finitely presented metabelian groups. In \cite{B-N-S} Bieri, Neumann and Strebel defined  the invariant  $\Sigma^1(G)$ for any finitely generated group $G$  and for 3-manifold groups they linked $\Sigma^1(G)$ with the Thurston norm. Higher dimensional homological invariants $\Sigma^n(G, A)$ for a $\Z G$-module $A$ were defined by Bieri and Renz in \cite{B-Renz}, where they showed  that   $\Sigma^n(G, \Z)$ controls which subgroups of $G$ that contain the commutator are of homological type $FP_n$. In \cite{Renzthesis} Renz defined  the higher dimensional homotopical invariant $\Sigma^n(G)$ for groups $G$ of homotopical type $F_n$ and similar to the homological case $\Sigma^n(G)$ controls the homotopical finiteness properties of the subgroups of $G$  above the commutator. In all cases the $\Sigma$-invariants are open subsets of the character sphere $S(G)$.
For a group $G$ of type $F_n$ we have $\Sigma^n(G) = \Sigma^n (G, \Z) \cap \Sigma^2(G)$.
The description of the $\Sigma$-invariants of right-angled Artin groups by Meier, Meinert and  Van Wyk shows  that  for $n \geq 2$ the inclusion $\Sigma^n(G) \subseteq \Sigma^n(G, \mathbb{Z})$ is not necessarily an equality \cite{Meinert-VanWyk}.

In general it is difficult to explicitly calculate the $\Sigma$-invariants of a group $G$ but for some special classes of groups and sometimes in particular small dimensions $n$ the $\Sigma^n(G)$  were studied: Thompson group $F$ \cite{B-G-K},  \cite{W-Z}, generalized Thompson groups $F_{n, \infty}$ \cite{Z},  metabelian groups \cite{B-Groves}, fundamental groups of compact K\"ahler manifolds \cite{D}, limit groups \cite{Desi1}, free-by-cyclic groups \cite{B-F},\cite{DK},  permutational wreath products \cite{Luis}, right-angled Artin groups (RAAGs) \cite{Meinert-VanWyk},  some Artin groups that are not RAAGs \cite{Kisnney}, \cite{Kisnney-K}.

 For a group $G$ we call a {\it character} a non-zero group homomorphism $G \to \mathbb{R}$.  The main result of this paper is the following theorem.

\begin{theorem}  \label{ThmI} Let $[H,H] \subseteq K \subseteq H$ be groups such that $H$ and $K$ are of type $FP_n$ (resp. $F_n$). Let $\chi : K \to \R$ be a character such that $\chi([H,H]) = 0$. Then $[\chi] \in \Sigma^n(K, \Z)$ (resp. $[\chi] \in \Sigma^n(K)$) if and only if $[\mu] \in \Sigma^n(H, \Z)$ (resp. $[\mu] \in \Sigma^n(H)$) for every character $\mu : H \to \R$ that extends $\chi$.
\end{theorem} 

Theorem \ref{ThmI} can be seen as a monoidal version of the Bieri-Renz result \cite[Thm.~B]{B-Renz} and its homotopical counterpart \cite[Satz~C]{Renzthesis}.  Furthermore our proof of Theorem \ref{ThmI}  is very much inspired by the proofs of \cite[Thm.~B]{B-Renz} and  \cite[Satz~C]{Renzthesis}. As an application of Theorem \ref{ThmI} we calculate $\Sigma^n(G, \mathbb{Z})$ and $\Sigma^n(G)$ for the Bestvina-Brady groups, in particular this applies to  the Bieri-Stallings groups.

The Bieri-Stallings groups $G_m$ were studied in \cite[Section~2]{Bieribook}, where Bieri proved that $G_m$ is of type $FP_{m-1}$ but not $FP_m$. The case $m = 3$ was previously considered by Stallings in \cite{S} and it was the first known example of a group that is finitely presented but not of type $FP_3$.  
 In \cite{KL2} Kochloukova and Lima  suggested  the Monoidal Virtual Surjection  Conjecture and they showed that it holds for $G_m$ in dimension $n$ provided $n \leq m-2$ and $m \geq 3$, i.e. they calculated $\Sigma^n(G_m, \Z)$ and $\Sigma^n(G_m)$. We will discuss in details the Monoidal Virtual Surjection  Conjecture in the preliminary section \ref{prel}.  As an application of Theorem \ref{ThmI} we complete in  Corollary  \ref{ThmJ} the calculation of $\Sigma^n(G_m, \Z)$  and $\Sigma^n(G_m)$ for any $n$.

  \begin{cor} \label{ThmJ} The Monoidal Virtual Surjection  Conjecture holds for the Bieri-Stallings groups $S = G_m$ where $m \geq 2$ and the  classical embedding $S = G_m \subseteq L_1 \times \ldots \times L_m$ with $L_i$ free of rank 2 for $ 1 \leq i \leq m$.
 \end{cor}

 The Bestvina-Brady groups $BB_{\Gamma}$ are defined as the kernels of homomorphisms $\varphi: A_{\Gamma} \to \BZ$ that take the usual generators to $1 \in \BZ$, where $A_{\Gamma}$ is the right-angled Artin group defined by the finite graph $\Gamma$. This is exactly the class of groups studied in \cite{B-B}, where it is shown that $BB_{\Gamma}$ can have various sets of exotic finiteness properties depending on the topological properties of the flag complex $\Delta$ associated to $\Gamma$, in particular $BB_{\Gamma}$ is $FP_n$ (resp. $BB_{\Gamma}$ is $F_n$) if and only if $\Delta$ is $(n-1)$-acyclic (resp. $\Delta$ is $(n-1)$-connected).
Theorem \ref{ThmI} implies that the $\Sigma$-invariants of $BB_{\Gamma}$ are determined by those of $A_{\Gamma}$ and this completes the description of  the $\Sigma$-invariants of $BB_{\Gamma}$ since the $\Sigma$-invariants of  $A_{\Gamma}$ are completely described in \cite{Meinert-VanWyk, MV}.  Note that the Bieri-Stallings group $G_m = BB_{\Gamma} $  and $L_1 \times \ldots \times L_m = A_{\Gamma}$ for the graph $\Gamma$ with vertices $x_1, y_1, \ldots, x_m, y_{m}$, where the only vertices  that are not connected by an edge in $\Gamma$ are $x_i$ and $y_i$ for $ 1 \leq i \leq m$. Here $L_i$ is a free non-abelian group generated by  $x_i$ and $y_i$.

\begin{cor} \label{cor-june} Let $\Gamma$ be a finite connected graph such that $\Delta$ is $(n-1)$-acyclic (resp. $\Delta$ is $(n-1)$-connected). Then $[\chi] \in \Sigma^n(BB_{\Gamma}, \mathbb{Z})$ (resp. $[\chi] \in \Sigma^n(BB_{\Gamma})$) if and only if $[\mu] \in \Sigma^n(A_{\Gamma}, \mathbb{Z})$ (resp. $[\mu] \in \Sigma^n(A_{\Gamma})$) for every character $\mu : A_{\Gamma} \to \mathbb{R}$ that extends $\chi$.	
\end{cor} 

 In \cite{B-Groves} using valuation theory Bieri and Groves showed that for a finitely generated metabelian group $G$ the complement $\Sigma^1(G)^c$ of $\Sigma^1(G)$ in the character sphere $S(G)$ is a rationally defined spherical polyhedron i.e. a finite union of finite intersections of rationally defined closed semispheres in $S(G)$, where the rationality  of the hemisphere means that it is defined by a rational vector. There are examples of groups $G$ of PL automorphism of the unit interval, where $\Sigma^1(G)^c$ contains precisely two points that need not be discrete \cite{B-N-S}, hence $\Sigma^1(G)^c$ is not a rationally defined polyhedron. We show that this cannot happen for the Bestvina-Brady group $BB_{\Gamma}$.

\begin{cor} \label{polyhedron} Let $\Gamma$ be a finite connected graph. Then $\Sigma^1( BB_{\Gamma}) ^c = S(BB_{\Gamma}) \setminus \Sigma^1(BB_{\Gamma})$ is a rationally defined spherical polyhedron. In particular for the  Bieri-Stallings groups $ G_m$, where $m \geq 2$, we have that $\Sigma^1(G_m)^c$ is a rationally defined spherical polyhedron.
\end{cor} 

We also strengthen  one of the directions of Theorem \ref{ThmI}.

\begin{theorem} \label{subnormal}
 Let $N \leq G$ be groups of type $FP_n$ and let $\chi: G \to \R$ be a character. Suppose that $N$ is a subnormal subgroup of $G$. If $[\chi|_N] \in \Sigma^n(N,\Z)$, then $[\chi] \in \Sigma^n(G,\Z)$.
\end{theorem}

This generalizes \cite[Proposition~C2.25]{Strebel}, where Strebel considered the case $n=1$. Notice that we do not require that the subgroups involved in some subnormal series for $N$ are of type $FP_n$. Such flexibility allows us, for instance, to apply the result to wreath products. 
Recall that for a group $G$ acting on a set $X$ and a group $H$ the (permutational restricted) wreath product $\Gamma = H \wr_X G$ is $(\oplus_{x \in X} H_x ) \rtimes G$, where each $H_x \simeq H$ and the $G$-action (via conjugation) on $\oplus_{x \in X} H_x $ permutes the copies of $H_x$ via the original $G$-action on $X$.

\begin{corollary}  \label{wreath}
 Let $\Gamma = H \wr_X G$ be a (permutational restricted) wreath product of type $FP_n$ and let $\chi: \Gamma \to \R$ be a character. Let 
 \[ T_{\chi} = \{ x \in X \mid \chi |_{H_x} \neq 0\}.\]
 If $T_{\chi}$ has at least $n+1$ elements, then $[\chi] \in \Sigma^n(\Gamma,\Z)$.
 \end{corollary}

Corollary \ref{wreath} generalizes previous work by Mendon\c{c}a in \cite{Luis} on the low dimensional (homotopical) invariants of wreath products. This also complements \cite[Theorem~8.1]{BCK}, where the set $T_{\chi}$ is assumed to be empty.  

This article is organized as follows: Sections \ref{prelimsigma} and \ref{prel} contain preliminaries about $\Sigma$-invariants and the special classes of groups that we consider. We prove Theorem \ref{subnormal} and Corollary \ref{wreath} in Section \ref{subnormalwreath}. In  Sections \ref{ap1} and \ref{ap2} we apply Theorem \ref{ThmI} to the particular cases. Finally, the main result is proved in Sections \ref{pf1} and \ref{pf2}.

\section{Preliminaries on the $\Sigma$-invariants}  \label{prelimsigma}

Let $G$ be a finitely generated group. By definition a character $\chi : G \to \R$ is a non-zero homomorphism and the character sphere $S(G)$ is the set of equivalence classes $[\chi]$ of characters $\chi : G \to \R$, where 
two characters $\chi_1$ and $\chi_2$ are equivalent if one is obtained from the other by multiplication with any positive real number. For a fixed character $\chi : G \to \R$ define
$$
G_{\chi} = \{ g \in G \mid \chi(g) \geq 0 \}.
$$

 If not stated otherwise the modules considered in this paper are left ones. 
Recall that for an associative ring $R$ and an $R$-module $A$ we say that $A$ is of type $FP_n$ over $R$ if $A$ has a projective resolution over $R$ where all projectives in dimension up to $n$ are finitely generated i.e. there is an exact complex
$$
{\mathcal P} : \ldots \to P_i \to P_{i-1} \to \ldots \to P_0 \to A \to 0,
$$
where each $P_j$ is a projective $R$-module and for $i \leq n$ we have that $P_i$ is finitely generated.

 Let  $D$ be an integral domain.
By definition for a   $DG$-module $A$ $$\Sigma^n_D(G, A) = \{[\chi] \in S(G) \mid A \textrm{ is of type $FP_n$ as $DG_{\chi}$-module}\}.$$ When $A$ is the trivial (left) $DG$-module $D$, we denote by $ \Sigma^n(G, D)$ the invariant $ \Sigma^n_D(G, D)$. When $D=A= \Z$, we recover the original BNS invariant $\Sigma^1(G)$ of \cite{B-N-S}, i.e. $\Sigma^1(G) = \Sigma^1(G,\Z)$.

Later we will need the description of $\Sigma^1(G)$ given by the Cayley graph of  a finitely generated group $G$. Let $X$ be a finite generating set of $G$. Consider the Cayley graph $\Gamma$ of $G$ associated with the generating set $X$ i.e. the set of vertices is $V(\Gamma) = G$ and the set of edges is  $E(\Gamma) = X \times G$ with the edge $e = (x, g)$ having beginning $g$ and end $g x$. The group $G$ acts on $\Gamma$ via left multiplication on $V(\Gamma)$ and $h. e = (x, hg) $ for any $h \in G$. The letter $x$ is called the label of the edge $e$ and we write $(x^{-1} , gx)$ for the inverse of $e$ and call $x^{-1}$ the label of $e^{-1}$. For a fixed character $\chi : G \to \R$ we write $\Gamma_{\chi}$ for the subgraph of $\Gamma$ spanned by the vertices in $G_{\chi}$. By definition
$$
\Sigma^1(G) = \{ [\chi] \in S(G) \mid \Gamma_{\chi} \hbox{ is a connected graph}\}.
$$

We can define similarly the invariant $\Sigma^2(G)$. Suppose that $G$ is finitely presented and let $\mathcal{C}$ be the Cayley complex of $G$ associated with some finite presentation $G = \langle X \mid R \rangle$. This is the $2$-complex obtained from the Cayley graph, as above, by gluing the set of $2$-dimensional cells $R \times G$, where the cell associated to $(r,g)$ is glued along the boundary described by the loop with label $r$ and base point $g$ in $\Gamma$. For any character $\chi: G \to \R$, the subset $G_{\chi} \subset G$ determines a full subcomplex $\mathcal{C}_{\chi}$ of $\mathcal{C}$. By definition
$$
\Sigma^2(G) = \{ [\chi] \in S(G) \mid \mathcal{C}_{\chi} \hbox{ is 1-connected for some finite presentation } \langle X|R\rangle  \hbox{ of } G\}.
$$

Finally, the higher homotopical invariants $\Sigma^n(G)$ can be defined for any group of type $F_n$ via the equality $\Sigma^n(G) = \Sigma^2(G) \cap \Sigma^n(G,\Z)$ for all $n \geq 2$. 

The first result is folklore, it is an obvious corollary of the fact that $\Sigma^1(G, \Z) = \Sigma^1(G)$ and tensoring is a right exact functor.

\begin{lemma} \label{quotient-sigma1} Let $\pi : G_1 \to G_2$ be an epimorphism of finitely generated groups, $\mu_2 : G_2 \to \mathbb{R}$ be a character (i.e. non-zero homomorphism) and $\mu_1 = \mu_2 \circ \pi$. Suppose that $[\mu_1] \in \Sigma^1(G_1)$. Then $[\mu_2] \in \Sigma^1(G_2)$.
\end{lemma} 

\begin{theorem} \cite[Thm.~ 9.3]{Meinert-VanWyk} \label{teo-Sigma.de.gr.=.Sigma.de.subgr.}
	Let $H$ be a subgroup of $G$, $A$ be a $DG$-module and $\xi: G \rightarrow \R$ be a character. If $[G:H] < \infty$ then $$[\xi|_H] \in \Sigma^n_D(H,A) \Leftrightarrow [\xi] \in \Sigma^n_D(G,A).$$
	In particular, if $n=0$, then $$A \textrm{ is a finitely generated } DG_{\xi} \textrm{-module } \Leftrightarrow  A \textrm{ is a finitely generated } DH_{\xi|_H} \textrm{-module}.$$
\end{theorem}

In \cite{B-G} Bieri and Geoghegan proved  a formula for the homological  invariants $\Sigma^n( - , F)$ for a direct product of groups, where $F$ is the trivial module and $F$ is a field.  If $F$ is substituted with the trivial module $\Z$ the result is wrong in both homological and homotopical settings provided that the dimension is sufficiently high, see \cite{Meinert-VanWyk} and \cite{Schutz}. 

\begin{theorem}  {\bf Direct product formula} \cite[Thm.~ 1.3, and Prop.~ 5.2]{B-G} \label{teo-conjec.prod.dir.corpo}
	Let $n \geq 0$ be an integer, $G_1, G_2$ be finitely generated groups and $F$ be a field. Then, $$\Sigma^n(G_1 \times G_2, F)^c = \displaystyle\bigcup_{p=0}^n \Sigma^p(G_1, F)^c * \Sigma^{n-p}(G_2, F)^c,$$ where $*$ denotes the join of sets in $S(G_1 \times G_2)$ and $^c$ denotes the set-theoretic complement of subsets of a suitable character sphere.
\end{theorem}

Note that by definition $\Sigma^0(G, F) = S(G)$.

The above theorem means that if $\mu : G_1 \times G_2 \to \R$ is a character with $\mu_1 = \mu \mid_{G_1}$ and  $\mu_2 = \mu \mid_{G_2}$ then $[\mu] \in \Sigma^n(G_1 \times G_2, F)^c = S(G_1 \times G_2, F) \setminus \Sigma^n(G_1 \times G_2, F)$ precisely when one of the following conditions hold :  

1. $\mu_1 \not=0$, $\mu_2 \not= 0$ and $[\mu_1] \in \Sigma^p(G_1, F)^c = S(G_1) \setminus \Sigma^p(G_1, F), [\mu_2] \in \Sigma^{n-p} (G_2, F)^c = S(G_2) \setminus \Sigma^{n-p}(G_2, F)$ for some $0 \leq p \leq n$;

or 

2. one of the characters $\mu_1$, $\mu_2$ is trivial and for the non-trivial one, say $\mu_i$, we have $[\mu_i] \in \Sigma^n(G_i, F)^c = S(G_i) \setminus \Sigma^n(G_i, F)$.

\begin{theorem} \cite{G}  \label{teo-conjec.prod.dir.corpo2}
	Let $n$ be a positive integer and let $G_1, G_2$ be finitely generated groups. If $1 \leq n \leq 2$, then
$$\Sigma^n(G_1 \times G_2, \Z)^c = \displaystyle\bigcup_{p=0}^n \Sigma^p(G_1, \Z)^c * \Sigma^{n-p}(G_2, \Z)^c,$$ where $*$ denotes the join of sets of the $S(G_1 \times G_2)$ and $^c$ denotes the set-theoretic complement of subsets of a suitable character sphere. Furthermore, the inclusion $``\subseteq"$ holds for all $n$.
\end{theorem}

This implies for instance that, for a direct product $P=G_1 \times \cdots \times G_{n+1}$ of type $FP_n$, any character $\chi: P \to \R$ such that $\chi |_{G_i} \neq 0 $ for all $1 \leq i \leq n+1$ represents an element of $\Sigma^n(P,\Z)$.

\begin{theorem} \label{BRhomotopic} \cite{B-Renz}, \cite{Renzthesis} Let $G$ be a group of type $F_n$ (resp. $FP_n$) and $N$ be a subgroup of $G$ that contains the commutator subgroup $G'$. Then $N$ is of type $F_n$ (resp. $FP_n$) if and only if
	$$S(G,N) = \{ [\chi] \in S(G) \mid \chi(N) = 0 \} \subseteq \Sigma^n(G) \ \ (\hbox{resp. } \Sigma^n(G,\Z)) .$$
\end{theorem}
Bieri and Renz proved the
 homological version of Theorem \ref{BRhomotopic} in \cite{B-Renz}. The homotopical version for $n = 2$ was proved by Renz in \cite{Renzthesis} and the general homotopical case $ n  \geq 3$ follows from the formula $\Sigma^n(G) = \Sigma^n(G, \mathbb{Z}) \cap \Sigma^2(G)$.

The following theorem can be traced back to several papers : Gehrke results in \cite{G}; the Meier, Meinert and VanWyk description of the $\Sigma$-invariants for right-angled Artin groups \cite{Meinert-VanWyk} or the Meinert result on the $\Sigma$-invariants for direct products of virtually free groups \cite{Me2}.

\begin{theorem} \label{new111} \cite{G}, \cite{Meinert-VanWyk}, \cite{Me2} If $\chi : F_2^s = F_2 \times \ldots \times F_2 \to \mathbb{R}$ is a character whose restriction on precisely $n$ copies of $F_2$ is non-zero, then 	$[\chi] \in \Sigma^{n-1}(F_2^s) \setminus \Sigma^{n}(F_2^s)  $. 
\end{theorem}

\section{Preliminaries on subdirect products, limit groups, the Virtual Surjection Conjecture and the Monoidal Virtual Surjection  Conjecture} \label{prel}

The class of limit groups contains all finite rank free groups and the orientable surface groups. It coincides with the class of the  finitely generated fully residually free groups $G$ i.e. for every finite subset $X$ of $G$ there is a free group $F$ and a homomorphism $\varphi : G \to F$ whose restriction on $X$ is injective. Limit groups are of homotopical type $F$ i.e. they are of type $FP_{\infty}$, finitely presentable and of finite cohomological dimension.
 
 Limit groups were used in the solution of the Tarski problem on the elementary theory of non-abelian free groups of finite rank obtained independently by Kharlampovich and Myasnikov  and by Sela in \cite{K-M}, \cite{Se}. By a result of Baumslag, Myasnikov and Remeslennikov in \cite{B-M-R} every finitely generated residually free group is a subdirect product of finitely many  limit groups.

A subgroup $G \subseteq G_1 \times \ldots \times G_m$ is a subdirect product if the projection map $p_i : G \to G_i$ is surjective for all $ 1 \leq i \leq m$. Denote by $p_{i_1, \ldots, i_n} : G \to G_{i_1} \times \ldots \times G_{i_n}$ the projection map that sends $(g_1, \ldots, g_m)$ to $(g_{i_1}, \ldots, g_{i_n})$.

\begin{theorem} \label{tra-la-la} \cite{Desi1} Let $G \subseteq G_1 \times \ldots \times G_m$  be a subdirect product of non-abelian limit groups $G_1, \ldots, G_m$ such that $G \cap G_i \not= 1$ for every $1 \leq i \leq m$. Then if $G$ is of type $FP_n$ for  some $n \leq m$ then $p_{i_1, \ldots, i_n}(G)$ has finite index in $G_{i_1} \times \ldots \times G_{i_n}$ for every $ 1 \leq i_1 < \ldots < i_n \leq m$.
\end{theorem}

\begin{theorem} \cite{Desi1} Let $G$ be a non-abelian limit group. Then $\Sigma^1(G) = \emptyset$.
	\end{theorem}

 The  following conjecture  was formulated by Kuckuck in \cite{Benno}.

\medskip
\noindent
{\bf The Virtual Surjection Conjecture} \cite{Benno} {\it Let $G \subseteq G_1 \times \ldots \times G_m$  be a subdirect product of  groups $G_1, \ldots, G_m$ such that $G \cap G_i \not= 1$ for every $1 \leq i \leq m$ and each $G_i$ is of homotopical type $F_n$ for a fixed $n \leq m$. Suppose that $p_{i_1, \ldots, i_n}(G)$ has finite index in $G_{i_1} \times \ldots \times G_{i_n}$ for every $ 1 \leq i_1 < \ldots < i_n \leq m$. Then $G$ is of type $F_n$. }

\medskip The motivation behind 
the Virtual Surjection Conjecture is that it holds for $n = 2$ \cite{BHMS} and this  particular case  was established by Bridson, Howie, Miller and Short as a corollary of the 1-2-3 Theorem. Furthermore the Virtual Surjection Conjecture holds for any $n$ when $G$ contains $G_1' \times \ldots \times G_m'$ \cite{Benno}. A homological version of the Virtual Surjection Conjecture was suggested in \cite{KL1} and proved for $n = 2$.

\begin{theorem} \cite{KL1}  \label{homological1-2-3} Let  $G \subseteq G_1 \times \ldots \times G_m$  be a subdirect product of  groups $G_1, \ldots, G_m$ such that $G \cap G_i \not= 1$ for every $1 \leq i \leq m$ and each $G_i$ is of homological type $FP_2$. Suppose that $p_{i_1, i_2}(G)$ has finite index in $G_{i_1} \times G_{i_2}$ for every $ 1 \leq i_1 < i_2 \leq m$. Then $G$ is of type $FP_2$.
	\end{theorem}

The following conjecture, called the Monoidal Virtual Surjection  Conjecture, was suggested in \cite{KL2} and should be viewed as a monoidal version of the Virtual Surjection Conjecture from \cite{Benno}.

\medskip
{\bf The Monoidal Virtual Surjection  Conjecture}  \cite{KL2} 

{\it Let $n$ and $m$ be positive integers such that $m \geq 2$ and $ 1 \leq n \leq m$. Let $S \leq L_1 \times \ldots \times L_m$ be a  subdirect product of non-abelian limit groups $L_1, \ldots, L_m$ such that $S$ is of type $FP_n$ and finitely presented.
	Then $$[\chi] \in \Sigma^n(S, \mathbb{Q}) = \Sigma^n(S, \mathbb{Z}) = \Sigma^n(S)$$ if and only if  $$p_{j_1, \ldots, j_n}(S_{\chi}) = p_{j_1, \ldots, j_n}(S) \hbox{ for all } 1 \leq j_1 < \ldots  < j_n \leq m,$$ where $p_{j_1, \ldots, j_n}: S \to L_{j_1} \times \ldots \times L_{j_n}$ is  the canonical projection.}

\medskip Note that by Theorem \ref{tra-la-la} the condition that  $S$ is of type $FP_n$ implies that $p_{j_1, \ldots, j_n}(S)$ has finite index in  $ L_{j_1} \times \ldots \times L_{j_n}$.

\medskip

\medskip We state several results that were proved recently by Kochloukova and Lima in \cite{KL2}.
\begin{theorem} \cite{KL2} \label{mono}
1. The forward direction of the Monoidal Virtual Surjection  Conjecture holds i.e. if $[\chi] \in \Sigma^n(S, \mathbb{Q})$ then $p_{j_1, \ldots, j_n}(S_{\chi}) = p_{j_1, \ldots, j_n}(S)$ for all $1 \leq j_1 < \ldots  < j_n \leq m$;

2. The Monoidal Virtual Surjection  Conjecture holds for $n = 1$;

 3. The Monoidal Virtual Surjection  Conjecture holds for $n = m$;

4. If the Virtual Surjection Conjecture  from \cite{Benno} holds then the Monoidal Virtual Surjection  Conjecture holds for all discrete characters $\chi$;

5. The Monoidal Virtual Surjection  Conjecture holds for the Bieri-Stallings groups $S = G_m$ and the  classical embedding $S = G_m \subseteq L_1 \times \ldots \times L_m$ with $L_i$ free of rank 2 for $ 1 \leq i \leq m$  in the following two cases : 

a)  if $n \leq m-2$;

b) if $n = m-1$ and $\chi$ is a discrete character.
\end{theorem}

\section{Invariants of subnormal subgroups and wreath products} \label{subnormalwreath}

In this section we prove Theorem \ref{subnormal}. For this we consider a characterization of the homological $\Sigma$-invariants in terms of \emph{Novikov rings}.

Let $G$ be a group and $\chi: G \to \R$ a character. The Novikov ring of $G$ with respect to $\chi$ is the ring $\widehat{\Z G}_{\chi}$ of formal sums $\lambda= \sum_{g \in G} a_g g$, with $a_g \in \Z$, such that the set 
\[S_{\lambda,r} = \{g \in G \mid a_g \neq 0, \chi(g)<r\}\]
is finite for all $r \in \R$. Addition and multiplication in $\widehat{\Z G}_{\chi}$ extend the operations of the subring $\Z G \subseteq \widehat{\Z G}_{\chi}$.

\begin{theorem} \cite{Bieri}, \cite{Si}
Let $G$ be a group of type $FP_n$ and $\chi: G \to \R$ a character. Then $[\chi] \in \Sigma^n(G,\Z)$ if and only if $H_k(G, \widehat{\Z G}_{\chi})=0$ for all 
$k=0,1, \ldots, n$. 
\end{theorem}

We will employ the following result obtained by Sch\"utz.

\begin{lemma} \cite[Corollary~2.5]{Sc}  \label{Schu}
Let $G$ be a group of type $FP_n$ and $\chi: G \to \R$ a  character. If $[\chi]\in \Sigma^n(G,\Z)$, then $H_k(G, M)=0$ for all $0 \leq k \leq n$ and for all $\widehat{\Z G}_{\chi}$-modules $M$.   
\end{lemma}

 {\bf Proof of Theorem \ref{subnormal}}
 Suppose that $N \leq G$ is subnormal and let $N=N_0 \leq N_1 \leq \dots \leq N_r = G$ with $N_i \trianglelefteq N_{i+1}$ for all $i$. 
  Suppose that $\chi: G \to \R$ is a  character such that $[\chi |_N] \in \Sigma^n(N,\Z)$. We will show by induction on $i$ that $H_j(N_i, \widehat{\Z G}_{\chi}) = 0$ for $ 0 \leq j \leq n$. 
  The case $ i= 0$ is a particular case of Lemma \ref{Schu}. 
  
  For the induction we assume that $H_j(N_i, \widehat{\Z G}_{\chi}) = 0$  for $ 0 \leq j \leq n$ and aim to prove that $H_j(N_{i+1}, \widehat{\Z G}_{\chi}) = 0$ for $ 0 \leq j \leq n$ if $i \not= r$.
  Consider the LHS spectral sequence
 \[ E_{p,q}^2 = H_p(N_{i+1}/N_i, H_q(N_i, \widehat{\Z G}_{\chi})) \Rightarrow H_{p+q}(N_{i+1}, \widehat{\Z G}_{\chi}).\] 
 Lemma \ref{Schu} implies that $E_{p,q}^2=0$ whenever $q \leq n$. By the convergence of the spectral sequence, $H_j(N_{i+1}, \widehat{\Z G}_{\chi})=0$ for all $j \leq n$. This completes the induction. 
  
  Finally since $N_r = G$ we deduce that $H_j(G,\widehat{\Z G}_{\chi})=0$ for all $0 \leq j \leq n$, that is, $[\chi] \in \Sigma^n(G,\Z)$.
 
 \vspace{3mm}
 
 {\bf Proof of Corollary \ref{wreath}} Let $\Gamma= H \wr_X G$ be a wreath product of type $FP_n$ and let $\chi:\Gamma \to \R$ be a  character with $|T_{\chi}| \geq n+1$. Let $T_1$ be any subset of $T_{\chi}$ containing exactly $n+1$ elements and let $N \leq \Gamma$ be the subgroup generated by the copies $H_x$ of $H$ such that $x \in T_1$. Clearly $N = \oplus_{x \in T_1} H_x \simeq \prod_{j=1}^{n+1} H$. Notice that $N$ is of type $FP_n$, since the direct product is finite and, under the conditions of the corollary, $H$ is of type $FP_n$ by \cite[Theorem~A]{BCK}. By Theorem \ref{teo-conjec.prod.dir.corpo2} we conclude that $[\chi |_N] \in \Sigma^n(N,\Z)$. Furthermore, $N$ is subnormal in $\Gamma$ via $N \leq \oplus_{x \in X}H_x \leq \Gamma$, so $[\chi] \in \Sigma^n(\Gamma, \Z)$ by Theorem \ref{subnormal}.

\section{Applications of Theorem \ref{ThmI} for the Bieri-Stallings groups $G_m$ : proof of  Corollary  \ref{ThmJ}}  \label{ap1}

Let $L_i$ be the free group with a free basis $\{ x_i, y_i \} $ and 
$$
G_m \leq H = L_1 \times L_2 \times \ldots \times L_m
$$
be the kernel of the character $H \to \R$ that sends each $x_i$ and $y_i$ to 1.
 Consider the Monoidal Virtual Surjection Conjecture for $S = G_m$ of type $FP_n$.
Note that since $G_m$ is $FP_{m-1}$ but not  $FP_m$ we  have that $n \leq m-1$. The case $n \leq m-2$ was proved in \cite{KL2} i.e. we know that  the  Monoidal Virtual Surjection Conjecture  holds for $S = G_m$ and $ n \leq m-2$. Thus we can assume from now on that $n = m-1$. Furthermore as the forward direction of the  Monoidal Virtual Surjection Conjecture  holds (see Theorem \ref{mono}, part 1), it remains to prove the backward direction, i.e. for $S = G_m$ if $p_{j_1, \ldots, j_{m-1}}(S) = p_{j_1, \ldots, j_{m-1}}(S_{\chi})$ for every $ 1 \leq j_1 < \ldots < j_{m-1} \leq m$ then $[\chi] \in \Sigma^{m-1}(S)$. Note that this is really sufficient to complete the proof of the Monoidal Virtual Surjection Conjecture  since in general we have $\Sigma^{m-1}(S) \subseteq \Sigma^{m-1}(S, \mathbb{Z}) \subseteq \Sigma^{m-1}(S, \mathbb{Q})$.

Let $\mu : H \to \R$ be a character and $\mu_i$ be the restriction of $\mu$ on  $L_i$. By Theorem \ref{new111}
$[\mu] \in \Sigma^{m-1}(H)  $ if and only if  $\mu_i \not= 0$ for all $ 1 \leq i \leq m$. Note that $[G_m, G_m] = [H,H]$, hence we can apply Theorem \ref{ThmI} i.e. $[\chi] \in \Sigma^{m-1}(S)$ if and only if for every character  $\mu : H \to \R$ that extends $\chi$ we have $[\mu] \in \Sigma^{m-1}(H)$.
Thus to finish the proof of the fact that $[\chi] \in \Sigma^{m-1}(S)$ for $S = G_m$ it remains to show that  $\mu_i \not= 0$ for all $ 1 \leq i \leq m$.

By  \cite{KL2} 
$p_{j_1, \ldots, j_{m-1}}(S_{\chi}) = p_{j_1, \ldots, j_{m-1}}(S)$  is equivalent to $\chi ( Ker (p_{j_1, \ldots, j_{m-1}})) \not= 0.$ 
Note that for $S = G_m$ we have $$Ker (p_{j_1, \ldots, j_{m-1}}) = G_m \cap L_i,$$ where $i = \{ 1, \ldots, m \} \setminus \{j_1, \ldots, j_{m-1}\}$. Thus $\chi(G_m \cap L_i) \not= 0$ and hence $\mu_i \not= 0$ for $ 1 \leq i \leq m$.
 This completes the proof of  Corollary \ref{ThmJ}.

\section{Applications of Theorem \ref{ThmI} to Bestvina-Brady groups : proof of Corollary \ref{cor-june} and  Corollary \ref{polyhedron}}  \label{ap2}

Let $A_{\Gamma}$ be the right-angled Artin group associated to a finite graph $\Gamma$.
Let $BB_{\Gamma}$ be the Bestvina-Brady group associated to  $\Gamma$ i.e. $BB_{\Gamma}$ is the kernel of the homomorphism $\varphi: A_{\Gamma} \to \Z$ that sends each of the usual generators of $A_{\Gamma}$ (which are identified with the vertices of $\Gamma$) to $1 \in \Z$. By \cite{B-B} we know that $BB_{\Gamma}$ is finitely generated if and only if $\Gamma$ is connected.

\begin{lemma} \label{june14}
 If $\Gamma$ is connected, then $[A_{\Gamma}, A_{\Gamma}] = [BB_{\Gamma}, BB_{\Gamma}]$.
\end{lemma}

\begin{proof}
 It suffices to show that for any two vertices $x,y \in V(\Gamma)$, the commutator $[x,y]$ lies in $[BB_{\Gamma}, BB_{\Gamma}]$. We prove this by induction on the length $n$ of the shortest path joining $x$ and $y$ in $\Gamma$. If $n=1$ there is nothing to do. Suppose that $n>1$ and let $t \in V(\Gamma)$ be a vertex which is connected to $y$ by an edge and which is connected to $x$ through a path of length $n-1$. By induction we can assume that $[x,t] \in [BB_{\Gamma}, BB_{\Gamma}]$. Keeping in mind that $y$ and $t$ commute in $A_{\Gamma}$, it is not hard to verify that
 \[ [x,y] = [xy^{-1}, t^{-1}y]^y . [x,t]^{t^{-1}y},\]
 thus $[x,y] \in [BB_{\Gamma}, BB_{\Gamma}]$. 
 \end{proof}

\subsection{Proof of  Corollary \ref{cor-june}} It follows from Lemma \ref{june14} that Theorem \ref{ThmI} applies to any non-zero homomorphism $\chi: BB_{\Gamma} \to \R$. Thus $[\chi]  \in \Sigma^m(BB_{\Gamma},\Z)$ if and only  if  $[\mu] \in \Sigma^m( A_{\Gamma},\Z)$ for any extension $\mu: A_{\Gamma} \to \R$, and similarly for the homotopical invariants. In this sense, the invariants of $A_{\Gamma}$ determine the invariants of $BB_{\Gamma}$.

 \subsection{Proof of Corollary \ref{polyhedron}}
Let $\Gamma$ be a connected graph with set of vertices $V(\Gamma)$.  Let $\chi : A_{\Gamma} \to \mathbb{R}$ be a  character and $L_{\chi}$ be the subgraph of $\Gamma$ spanned  by the vertices of $\Gamma$ with non-zero $\chi$-value. By
 \cite{MV}  $[\chi] \in \Sigma^1(A_{\Gamma})$ if and only if $L_{\chi}$ is connected  and dominant in $\Gamma$ i.e. every vertex in $V(\Gamma) \setminus V(L_{\chi})$ is linked by an edge from $\Gamma$  with at least  one vertex from $L_{\chi}$. This description easily implies that $\Sigma^1(A_{\Gamma})^c$ is a rationally defined spherical polyhedron. Actually there is a conjecture suggested in \cite{Kisnney-K} that predicts the structure of the $\Sigma^1$ for general Artin groups. It was recently shown in \cite{Koch-even} that for the  Artin groups that satisfy this conjecture $\Sigma^1(G)^c$ is a rationally defined polyhedron. In particular this holds for $G = A_{\Gamma}$. 

By Lemma \ref{june14} $S(BB_{\Gamma})$ is a subsphere of $S(A_{\Gamma})$. The fact that $\Sigma^1(A_{\Gamma})^c$ is a rationally defined spherical polyhedron together with Theorem \ref{ThmI} implies that $\Sigma^1(BB_{\Gamma})^c$ is a rationally defined spherical polyhedron.

\section{The proof of the homological part of Theorem \ref{ThmI}}  \label{pf1}

In this section we will prove the homological part of Theorem \ref{ThmI}. The proof is long and technical and subdivided in several subsections. We use the techniques developed in \cite{B-Renz} but at several points subtle modifications are required.

\subsection{Preliminaries from the Bieri-Renz paper \cite{B-Renz}}

First we recall some notation from \cite{B-Renz}.
Assume that  \begin{equation} \label{free}
 \ldots ~ \mapnew{\partial_{i+1}} ~ F_i ~ \mapnew{\partial_i} ~ F_{i-1} ~ \mapnew{\partial_{i-1}} ~ \ldots \to F_0 \to A \to 0\end{equation}
is an admissible free resolution (over $\Z G$) with
finitely generated $n$-skeleton i.e. for every $0 \leq i \leq n$
the free $\Z G$-module $F_i$ is endowed with a basis $X_i$  such that for each $x_i \in X_i$ we have $\partial_i (x_i) \not= 0$ and
$$
X(n) = \cup_{0 \leq i \leq n} X_i \hbox{ is finite.}
$$ Set $$X = \cup_{i \geq 0} X_i.$$ We write ${\bf F}$ for the complex $\ldots \to F_i \to F_{i-1} \to \ldots \to F_0 \to 0$ and write ${\bf F} \to A$ for the free resolution (\ref{free}). 
We think of ${\bf F}$ as $\oplus_{i \geq 0} F_i$.

Let $v  : X \to \R$ 
be a set-theoretic map and $$\nu : G \to \R$$ be a character. We extend it to a valuation (in the sense of Bieri-Renz in \cite{B-Renz}) $$v = v_{\nu}: {\bf F} \to \R \cup \{ \infty \},$$
given by
$$ v(0) = \infty, \ v(g x) =  \nu(g) + v(x) \hbox{ for } g \in G, x \in X$$and
$$v ( \sum_{z_{g,x} \in \Z \setminus \{ 0 \}, g \in G, x \in X} z_{g,x} g x) = min \{ v(g x) \mid z_{g,x} \not= 0 \}.
$$
The map $v \mid_{X_i} = v_i$ is defined inductively on $i \geq 0$ by
$$
v_i(x) = 0 \hbox{ if } i = 0 \hbox{ and } v_i(x) = v_{i-1} (\partial_i (x)) \hbox{ for } i > 0.
$$ 
Thus $$v(\partial c) \leq v(c)\hbox{ for }c \in {\bf F}.$$

We define  support $supp_X (c)$ with respect to $X$  for $c = \sum_{g \in G, x \in X_i}  z_{g,x} g x \in F_i \setminus \{ 0 \}$, where $z_{g,x} \in \Z \setminus \{ 0 \}$, as a finite subset of $G$ given by
$$
supp_X (c) = \cup_{g \in G, x \in X_i, z_{g,x} \not= 0}  \ g. supp_X(x) \hbox{ for } i \geq 0,
$$
$$supp_X(x) = supp_X(\partial_i(x)) \hbox{ for } x \in X_i, i \geq 1
$$
and 
$$supp_X(x) = 1_G\hbox{ for } x \in X_0.
$$
For completeness we set  $supp_X(0) = \emptyset$.

 Note that 
 for $c \in F_i$ we have that
 $$supp_X(\partial(c)) \subseteq supp_X(c)$$
 and if $a,b \in F_i$  then  
\begin{equation}  \label{inclusion}
supp_X( a+ b ) \subseteq supp_X(a) \cup supp_X(b).\end{equation}
If $a,b, a+ b \in F_i \setminus \{ 0 \}$, $a = \sum_{g \in G, x \in X_i}  z_{g,x} g x$, $b = \sum_{g \in G, x \in X_i}  \widetilde{z}_{g,x} g x$ and $$\{ (g,x) \in G \times X_{ i} ~| ~ z_{g,x} \not= 0 \} \cap \{ (g,x) \in G \times X_{ i} ~| ~ \widetilde{z}_{g,x} \not= 0 \} = \emptyset$$ then in (\ref{inclusion}) equality is attained.

We state an important result of Bieri and Renz that works for  compact subsets  of $S(G)$.

\begin{theorem} \cite{B-Renz}  \label{BR-compact}
	Let ${\bf F} \to A$ be an admissible free resolution with finitely
	generated $n$-skeleton. Then the following three conditions are equivalent for a
	compact subset $\Gamma \subseteq S(G)$:
	
	(i) $\Gamma \subseteq \Sigma^n(G, A)$;
	
	(ii) there is a finite set $\Phi$ of chain endomorphisms $\varphi : {\bf F} \to  {\bf F},$ lifting $id_A$, with
	the property that for each point $[\chi] \in \Gamma$ 
	there is some
	$\varphi \in \Phi$ with
	$$ v_{\chi}( \varphi(x)) > v_{\chi}(x) \hbox{ for every } x \in X(n);$$
	
	(iii) after replacing ${\bf F}$ by a suitable admissible free resolution, obtained by
	performing on ${\bf F}$ a finite sequence of elementary expansions, we can find a finite set $\Phi$ as
	in (ii) and for each $\varphi \in \Phi$  a chain homotopy $\sigma : \varphi \to Id_{\bf F}$  with $\sigma(X_i) \subseteq X_{i+1
	} \cup  \{ 0 \}$
	for every $i$ with $0 \leq i \leq n$.
\end{theorem}

\subsection{Proof of  the homological part of  Theorem \ref{ThmI}}  \label{difficult}

 Notice that Theorem \ref{subnormal} already proves one of the directions of Theorem \ref{ThmI} in the homological case. In this subsection we prove the other direction.

\begin{theorem}  \label{diffthm} Let $[H,H] \subseteq K \subseteq H$ be groups such that $H$ and $K$ are of type $FP_n$. Let $\chi : K \to \R$ be a character such that $\chi([H,H]) = 0$.  Suppose that $[\mu] \in \Sigma^n(H, \Z)$ for every character $\mu : H \to \R$ that extends $\chi$. Then $[\chi] \in \Sigma^n(K, \Z)$.
\end{theorem}

\begin{proof} 
Let $H_0$ be a subgroup of finite index in $H$ that contains the commutator $[H,H]$ such that 	$H_0/ [H,H]$ is torsion-free and for $K_0= H_0 \cap K$ we have $H_0/ K_0$ is torsion-free. Thus $H_0/ [H,H] \simeq K_0/ [H,H] \times H_0/ K_0 \simeq \mathbb{Z}^m \times \mathbb{Z}^{k-m} \simeq \mathbb{Z}^k$.

Set $$Q_0 = H_0/ [H,H], \  V = Q_0 \otimes_{\Z} \R \simeq \R^k \hbox{ and } W = K_0 / [H,H] \otimes_{\Z} \R \subseteq V.$$ We identify $W$ with a subspace of $V$. We  identify $V$ with $\R^k$, $Q_0$ with $\Z^k$ and consider the standard inner product $\langle - , - \rangle$ in $\R^k$.

By Theorem \ref{teo-Sigma.de.gr.=.Sigma.de.subgr.} for a character $\mu : H \to \R$ we have that $$[\mu] \in \Sigma^n(H, \Z)\hbox{ if and only if }[\mu_0] \in \Sigma^n(H_0, \Z),$$ where $\mu_0$ is the restriction of $\mu$ on $H_0$. Now the same holds for the character $\chi  : K \to \R$ from the statement and its restriction $\chi_0 : K_0 \to \R$ i.e. $$[\chi] \in \Sigma^n(K, \Z) \hbox{ if and only if }[\chi_0] \in \Sigma^n(K_0, \Z).$$ 
	Thus to prove the result we assume that for each character $\mu_0 : H_0 \to \R$ that extends the character $\chi_0$ we have \begin{equation} \label{new-line1} [\mu_0] \in \Sigma^n(H_0, \Z) \end{equation}  and aim to prove that
	$[\chi_0] \in \Sigma^n(K_0, \Z)$.
	
	Note that there are unique $c_{\mu_0} \in V$, $c_{\chi_0} \in W$ such that $$\mu_0(h) = \langle \pi(h), c_{\mu_0} \rangle \hbox{ for }h \in H_0$$ and $$\chi_0(g) = \langle \pi(g), c_{\chi_0} \rangle \hbox{  for }g \in K_0,$$ where $$\pi : H_0 \to Q_0$$ is the canonical projection.  Thus $c_{\chi_0}$ is the orthogonal projection of $c_{\mu_0}$ to the subspace $W$.
	
	Consider 
	$$
	W_{\chi_0 \geq 0} = \{ v \in W \mid \langle v,  c_{\chi_0} \rangle \geq 0 \}
	$$
	and
	$$
	W_{\chi_0 = 0} = \{ v \in W \mid \langle v,  c_{\chi_0} \rangle =  0 \}.
	$$
	Define
	$$
	D_r = \cup_{w \in 	W_{\chi_0 \geq 0}} B(w, r),
	$$
	where $B(w,r)$ is the closed ball in $V$ with center $w$ and radius $r$.
	Write 
	$$
	B(w,r)_{\chi_0 \leq 0} = \{ v \in B(w,r) \mid \langle v - w, c_{\chi_0} \rangle \leq 0 \}.
	$$
	Let $S(w,r)$ be the boundary of the closed ball $B(w,r)$, i.e. $S(w,r)$ is the sphere with center $w$ and radius $r$.
Write $$\delta : V \to W	$$for the orthogonal projection.
	Consider the boundary $B_r$ of $D_r$
	$$
	B_r = \partial D_r = 
	(\cup_{w \in 	W_{\chi_0 = 0}} \{ v \in S(w,r) \mid \delta(w - v) = \lambda c_{\chi_0} \hbox{ for } \lambda \in \R_{>0}  \}) 
	$$ $$\cup(\cup_{w \in 	W_{\chi_0 \geq 0}} \{v \in S(w,r) \mid \delta(v) = w \}).$$

	We start by fixing an admissible free resolution with finite $n$-skeleton $${\bf F} \to \Z$$  of the trivial $\Z H_0$-module $\Z$. Consider the compact set 
	$\Gamma$  $$\Gamma = \{ [\mu_0] \in S(H_0)  \mid \mu_0 |_{K_0} = \chi_0 \} \cup S(H_0, K_0)    \subseteq \Sigma^n(H_0, \Z) \subseteq  S(H_0),
	$$
	where the first inclusion uses (\ref{new-line1}) and Theorem \ref{BRhomotopic}. Thus
	we can apply Theorem \ref{BR-compact} i.e. there is a finite set $\Phi$ of chain endomorphisms $$\varphi : 
	{\bf F} \to {\bf F}$$ lifting $id_{\Z}$ such that for every
	$[\nu] \in \Gamma$ there is some $\varphi \in \Phi$ with
	\begin{equation} \label{importante1}
	v_{\nu} (\varphi(x)) > v_{\nu}(x) \hbox{ for every } x \in X(n)
	\end{equation} 
	and chain homotopy $$\sigma_{\varphi} : \varphi \simeq id_{\bf F} \hbox{ for each }\varphi \in \Phi$$  such that
	\begin{equation} \label{importante2} \sigma_{\varphi} (X_i) \subseteq X_{i+1} \cup \{ 0 \} \hbox{ for every } 0 \leq i \leq n.
	\end{equation}

	Similarly to \cite{B-Renz}  we define 
	\begin{equation} \label{def-s} s = max \{ || \pi( \sigma_{\varphi} (y)) || ~ \mid
	~  \varphi  \in \Phi,  ~ 1_{ H_0}\in supp_X(y), ~ y \in H_0 X(n)  \}. \end{equation}
Note that $s$ is well-defined since both the sets $X(n)$ and   $\{  y \in H_0 X(n) ~| ~ 1_{ H_0} \in supp_X(y) \}$ are finite.
	
	 Since $\Gamma$ is compact and both $X(n)$ and $\Phi$ are finite there is a real number \begin{equation} \label{def-t} t = 
		min_{x \in X(n), \varphi \in \Phi, [\nu] \in \Gamma}(v_{\nu}(\varphi(x) ) - v_{\nu}(x)) > 0.\end{equation}	Then by \cite[Lemma~4.3]{B-Renz} for $\lambda \in \cup_{j \leq n} F_j$  and $\varphi \in \Phi$, $[\nu] \in \Gamma$ we have \begin{equation} \label{use-t}
	supp_X \sigma_{\varphi} (\lambda) \subseteq supp_X( \lambda) \cup (H_0)_{\nu \geq v_{\nu}(\lambda) + t}. \end{equation}	
	Set the constants 
	\begin{equation} \label{def-constants}  	\epsilon_0 = \frac{1}{2} min \{ s , t  \}\hbox{   and }  a =  max \{ \frac{s^2 - \epsilon_0^2}{2 (t - \epsilon_0)}, \epsilon_0 \}.
 \end{equation} 
		\medskip
	Let $\Delta$ be a finite subset of $H_0$ and recall that $\pi : H_0 \to Q_0 = H_0/ [H,H] \subset V
	$ is the canonical projection. Define
	\begin{equation} \label{Delta}
	\alpha(\Delta) = min \{ r \mid \pi(\Delta) \subseteq D_r \}.
	\end{equation} 
	
	We will prove later the following fact.
	
	\medskip
	{\bf Claim } {\it For every cycle $z \in F_{j-1} \setminus \{ 0 \}$  with $j \leq n$ such that $\alpha(supp_X(z)) \geq a$ there is $c_0 \in F_j$ with $\partial(c_0) = z$ and $\alpha(supp_X(c_0)) \leq  \alpha (supp_X(z))  $.}
	
	\medskip
	The Claim  implies that  for every $r \geq a$
	$$
	{\bf F}_r : = \{ f \in {\bf F} \mid \pi(supp_X(f)) \subseteq D_r \}$$ 
	is a subcomplex  of ${\bf F}$ that is exact  in dimensions $1, \ldots, n-1$  and ${\bf F}_r \to \mathbb{Z}$ is exact in dimension 0. This will complete the proof modulo the following argument.
	
	Define 
	$$
	W_{\chi_0 \geq j} = \{ v \in W \mid \langle v,  c_{\chi_0} \rangle \geq j \}
	$$
	and 
	\begin{equation} \label{needed}
	D_{r, \chi_0 \geq j} = \cup_{w \in 	W_{\chi_0 \geq j}} B(w, r),
	\end{equation}
	where $B(w,r)$ is the closed ball in $V$ with center $w$ and radius $r$. Thus $$D_r = 	D_{r, \chi_0 \geq 0}.$$ The proof of the Claim will show that for every $r \geq a$ and every $j \in \mathbb{R}$ we have
	$$
	{\bf F}_{r, \chi_0 \geq j} : = \{ f \in {\bf F} \mid \pi(supp_X(f)) \subseteq D_{r, \chi_0 \geq j} \}$$ 
	is a subcomplex  of ${\bf F}$ that is exact  in dimensions $1, \ldots, n-1$.
	
	The original Bieri-Renz argument from \cite{B-Renz}  gives as well that
	$$
	{\bf \widetilde{F}}_r : = \{ f \in {\bf F} \mid \pi_0(supp_X(f)) \subseteq \widehat{B}_r \}
	$$
		is a subcomplex  of ${\bf F}$ that is exact  in dimensions $1, \ldots, n-1$ and $H_0({\bf \widetilde{F}}_r \to \mathbb{Z})  = 0$, where $$\pi_0 : H_0 \to H_0/ K_0 \simeq \mathbb{Z}^{k-m}$$ is the canonical projection and $\widehat{B}_r$ is the closed ball in $H_0/ K_0 \otimes \mathbb{R} \simeq \mathbb{R}^{k-m}$  with radius $r$ and center the origin.
	Thus  $	{\bf \widetilde{F}}_{r}$ is a partial free resolution 	up to dimension $n$ of the trivial $\mathbb{Z} K_0$-module $\mathbb{Z}$ with all modules in dimensions up to $n$ being finitely generated.

	Observe that $$	{\bf \widetilde{F}}_r = \cup_{j \in \mathbb{R}} 	{\bf F}_{r, \chi_0 \geq j}. $$
	Let \begin{equation} \label{pi-1} \pi_1 : H_0 \to H_0/ [H,H] = (K_0/ [H,H]) \times (H_0/ K_0) \to K_0/ [H_0,H_0], \end{equation}  be the composition where the first map is the canonical projection  and the second map is the projection on the first factor. 
	Set
	$$
	{\bf \widetilde{F}}_{r, \chi_0 \geq j}  := \{ f \in 	{\bf \widetilde{F}}_r \mid \pi_1(supp_X(f)) \subseteq \chi_0^{-1}([j, \infty)) \}.
	$$	 Fix some $r \geq a$. Note that there is a fixed negative integer $r_0$ that depends on $r$ such that for every $j \in \mathbb{R}$
	$$
		{\bf \widetilde{F}}_{r, \chi_0 \geq j} \subseteq 	{\bf F}_{r, \chi_0 \geq j}  \subseteq 		{\bf \widetilde{F}}_{r, \chi_0 \geq  j+ r_0} \subseteq 	{\bf F}_{r, \chi_0 \geq j + r_0}.
		$$ 
	The above inclusions induce maps in homology 
	$$
	H_i( 	{\bf \widetilde{F}}_{r, \chi_0 \geq j} ) \to H_i({\bf F}_{r, \chi_0 \geq j}) = 0 \to H_i( {\bf \widetilde{F}}_{r, \chi_0 \geq j+  r_0} ) \hbox{ for }   1 \leq i \leq n-1  \hbox{ and }j \in \mathbb{R}$$ 
	 and $$
	H_0( 	{\bf \widetilde{F}}_{r, \chi_0 \geq j} \to \mathbb{Z} ) \to H_0({\bf F}_{r, \chi_0 \geq j} \to \mathbb{Z}) = 0 \to H_0( {\bf \widetilde{F}}_{r, \chi_0 \geq j+  r_0} \to \mathbb{Z}) \hbox{ for } j \in \mathbb{R}.$$ Hence
	\begin{equation} \label{B-R12-1}
	H_i( 	{\bf \widetilde{F}}_{r, \chi_0 \geq j} ) \to  H_i( {\bf \widetilde{F}}_{r, \chi_0 \geq  j + r_0} ) \hbox{ is the zero map for }   1 \leq i \leq n-1 \hbox{ and } j \in \mathbb{R}
	\end{equation} 
	 and
	\begin{equation} \label{B-R12-2}
	H_0( 	{\bf \widetilde{F}}_{r, \chi_0 \geq j} \to \mathbb{Z}) \to  H_0( {\bf \widetilde{F}}_{r, \chi_0 \geq  j + r_0} \to \mathbb{Z}) \hbox{ is the zero map for }  j \in \mathbb{R}.
	\end{equation} 
Recall that  $	{\bf \widetilde{F}}_{r}$ is a partial free resolution 	up to dimension $n$ of the trivial $\mathbb{Z} K_0$-module $\mathbb{Z}$ with all modules in dimensions up to $n$ being finitely generated. Then  (\ref{B-R12-1}) and (\ref{B-R12-2}) we can apply \cite[Thm.~3.2]{B-Renz} to deduce
$$[\chi_0] \in \Sigma^n(K_0, \Z).$$

\begin{inote}
 Notice that ${\bf \widetilde{F}}_{r, \chi_0 \geq j}$ is a filtration of ${\bf \widetilde{F}}_r$ that comes from the valuation $v$ such that $v(f) = min (\chi_0(\pi_1(supp_X(f)))$ for all $f \in {\bf \widetilde{F}}_r$. We could also consider a new  basis $\tilde{X}$ of ${\bf \widetilde{F}}_r$ as  complex of free $\mathbb{Z}K_0$-modules, which gives rise to a new Bieri-Renz valuation. The valuations are clearly equivalent (in the sense of \cite[Section~2.2]{B-Renz}), which justifies the application of \cite[Thm.~3.2]{B-Renz} (see the remark following it).
\end{inote}

	{\bf Proof of the Claim} Let $c \in F_j$ be any element such that $\partial(c) = z$. If $\alpha (supp_X(c)) \leq \alpha (supp_X(z)) $ we are done, so assume that  $$\alpha (supp_X(c)) >  \alpha (supp_X(z)) \geq a.$$ Let $g \in supp_X(c)$ such that
	$$
	\alpha( \{ g \}) = \alpha (supp_X(c)) = r > a.$$
	Thus
	$$\pi(g) \in \partial D_r = B_r.$$
	
	Then either
	
	1. 
	$\pi(g) \in S(w,r)$ for some $w \in W_{\chi_0 \geq 0},$ where $w$ is the orthogonal projection  $\delta(\pi(g))$ of $\pi(g)$ to $W$;
	
	or
	
	2. 	$\pi(g) \in S(w,r)$ for some $w \in W_{\chi_0 = 0},$ where	$w$ is the orthogonal projection of $\pi(g)$ to $W_{\chi_0 = 0}$ and $\delta( w - \pi(g)) = \lambda c_{\chi_0}$ for some $\lambda  \in \mathbb{R}_{>0}$.
	
	Consider the character  $$\nu : H_0 \to \R$$
	given by \begin{equation} \label{def-nu} \nu(h) = \langle \pi(h), \frac{w - \pi(g)}{|| w - \pi(g) ||} \rangle.\end{equation} Then in case 1 $\nu(K_0) = 0$ and in case 2 
	the restriction of $\nu$ to $K_0$ is $ \lambda \chi_0$ for some $\lambda \in \R_{>0}$.
		In both cases 
	$$
	[\nu] \in \Gamma \subseteq \Sigma^n(H_0, \Z).
	$$ 	
		For the people familiar with the Bieri-Renz paper \cite{B-Renz} this means that we can "push"  every element of $H_0 \cap B_r$ inside  $D_r$ using the pushing algorithm from \cite{B-Renz}. We will explain in details in the rest of the section what this means. 	
The aim of the proof is to find an element $\widetilde{c} \in F_j$ such that $$\partial_j(c) = z = \partial_j(\widetilde{c})$$ and
	$$
	supp_X(\widetilde{c}) \subseteq (supp_X(c) \setminus \{ g \}) \cup  D_{r - \epsilon_0},
	$$
	 where  $\epsilon_0$  was defined in (\ref{def-constants}). Applying this procedure several times  we get the desired element $c_0 \in F_j$  from the Claim.

	In order to construct $\widetilde{c}$ we decompose $c = \sum z_y y$, where $z_y \in \Z \setminus \{ 0 \}, y \in H_0 X_j$ and write
	$$c = c' + c'',$$
	where  $c'$  contains precisely the terms $z_y y$ for which $g \in supp_X(y)$. Thus  $y$ cannot appear simultaneously  in $c'$ and $c''$ and this guarantees that $$supp_X(c) = supp_X(c') \cup supp_X(c'').$$ Then  $$g \not\in supp_X(c''), ~  supp_X(c'') \subseteq supp_X(c) ~ \hbox{ and } ~  g \in supp_X(c') \subseteq supp_X(c).$$Fix $\varphi$, $\sigma_{\varphi}$ and $v_{\nu}$  as before, for simplicity we denote by $\varphi$, $\sigma$ and $v$.  Define
	\begin{equation} \label{def-new} 
	\widetilde{c} = c + \partial \sigma (c') = \varphi(c') - \sigma \partial (c') + c''.\end{equation}

	 Recall (\ref{use-t}), i.e. for $t$ defined in (\ref{def-t}) and  for every element $\lambda \in \cup_{j \leq n}  F_j$ we have  \begin{equation} \label{eq-t0}
	supp_X \sigma(\lambda) \subseteq supp_X( \lambda) \cup (H_0)_{\nu \geq v(\lambda) + t}. \end{equation}
	Hence 
	\begin{equation} \label{eq-t}
	supp_X(\sigma \partial (c'))  \subseteq supp_X( \partial (c')) \cup (H_0)_{\nu \geq v(\partial(c')) + t} \subseteq supp_X( \partial (c')) \cup (H_0)_{\nu \geq v(c') + t},
	\end{equation}
	 where the last inclusion follows from the fact that
		$$ supp_X( \partial (c') ) \subseteq supp_X( c' ).$$
	Since $\partial (c') = z - \partial (c'')$ we get
	\begin{equation} \label{eq-t2} supp_X( \partial (c') ) \subseteq supp_X(z) \cup supp_X( \partial (c'') ).
	\end{equation}
	Furthermore  
	\begin{equation} \label{eq-t3} supp_X( \partial (c'') ) \subseteq supp_X( c'' ).\end{equation} Thus combining (\ref{eq-t}),  (\ref{eq-t2}) and (\ref{eq-t3}) we obtain
	$$supp_X(\sigma \partial (c')) \subseteq (H_0)_{\nu \geq v(c') + t} \cup
	supp_X(z) \cup supp_X( c'' ).$$	
	Since $v(c')  = \nu(g)$ we have $g \notin 	(H_0)_{\nu \geq v(c') + t} $. Furthermore $\alpha(  supp_X(c)) = \alpha( \{ g \}) > \alpha ( supp_X(z))$ implies that $g \notin supp_X(z) $ and so
	$$g \notin
	(H_0)_{\nu \geq v(c') + t} \cup
	supp_X(z) \cup supp_X( c'' )
	$$ and we deduce that
	\begin{equation} \label{bus} g \notin supp_X(\sigma \partial (c')).\end{equation}	
	In addition by (\ref{eq-t0})
	$$supp_X ( \widetilde{c} ) \subseteq supp_X (c) \cup supp_X( \partial \sigma(c')) \subseteq supp_X (c) \cup supp_X(  \sigma(c')) \subseteq $$ $$ supp_X (c) \cup supp_X (c') \cup (H_0)_{ \nu \geq v(c') + t} =  supp_X (c) \cup (H_0)_{ \nu \geq v(c') + t} \subseteq $$ $$ supp_X (c) \cup (H_0)_{ \nu \geq v(c) + t}.
	$$
	Note that $v(c) = v(c') =  - \alpha(\{ g \})$ and $t > 0$. Hence
	\begin{equation} \label{eq-t4}
	supp_X ( \widetilde{c} ) \subseteq supp_X (c) \cup (H_0)_{ \nu \geq - \alpha(\{ g \}) + t}.
	\end{equation}
	Note that by (\ref{importante1})
	$$
	v(\varphi(c'))  > v(c') = - \alpha(\{ g \}),
	$$
	hence \begin{equation} \label{bus1} g \not\in supp_X(\varphi(c')).\end{equation} Furthermore by the definition of $c''$ we have $g \not\in supp_X(c'').$ Then by (\ref{def-new}), (\ref{bus}) and (\ref{bus1}) we obtain
	\begin{equation} \label{bus2} g \notin supp_X(\widetilde{c}).\end{equation}
	
	Now for every $y \in H_0 X_j$ in the decomposition $c' = \sum  z_y y$ we have $g \in supp_X(y)$, hence $$ 1_{ H_0} \in g^{-1} supp_X(y).$$ By the definition of $s$ in (\ref{def-s}) we have that  $$ || \pi(b) || \leq   s  \hbox{ for } b \in g^{-1} supp_X(  \sigma (c')). $$
	Then by the definition of $\widetilde{c}$ for every
	$$
	h \in supp_X(\widetilde{c}) \setminus supp_X(c) \subseteq supp_X( \partial \sigma (c')),$$  hence $$ g^{-1} h \in g^{-1} supp_X( \partial \sigma (c'))  \subseteq  g^{-1} supp_X(  \sigma (c')) $$ and by the above 
	$$ || \pi(g^{-1} h) || \leq s.$$
	Moving to additive notation in $V$ we have
	\begin{equation} \label{eq-h1} || \pi(h) - \pi(g) || \leq s. \end{equation}
	Recall that by the definition of $h$ we have $h \not\in supp_X(c)$ and furthermore by (\ref{eq-t}) 
	$$h \in supp_X(\widetilde{c}) \setminus supp_X(c) \subseteq supp_X(\partial \sigma(c')) \subseteq $$ $$ supp_X(\sigma(c')) \subseteq supp_X(c') \cup (H_0)_{\nu \geq v(c') + t}  \subseteq supp_X(c) \cup (H_0)_{\nu \geq v(c') + t},$$
	hence  
	$$h \in (H_0)_{\nu \geq v(c') + t}$$ and 
	\begin{equation} \label{eq-h2}
	\nu(h) \geq v(c') + t =  - \alpha(\{ g \}) + t = - r + t.\end{equation}
	Recall that $g$ belongs to the boundary of the sphere with
	radius $r = \alpha( \{ g \})$ and center $w$ and $\nu$ is defined by (\ref{def-nu}).
	The conditions (\ref{eq-h1}) and (\ref{eq-h2}) imply that $h$ belongs to the sphere with radius $r_0 $ and center $w$, where  by the Pythagorean theorem applied twice
	$$
	r_0^2 = s^2 - t^2 + (r-t)^2 = r^2 - 2 r t + s^2.
	$$
	Let $$\epsilon_0 = \frac{1}{2} min \{ s , t  \}.$$
	Then $r_0 \leq r - \epsilon_0$ is equivalent to
	$r  \geq \epsilon_0$ and
	$
	r^2 - 2 r t + s^2 \leq (r - \epsilon_0)^2 = r^2 - 2 r \epsilon_0 + \epsilon_0^2.
	$
	The last is equivalent to 
	$
	r \geq \frac{s^2 - \epsilon_0^2}{2 (t - \epsilon_0)}.
	$
	Thus it suffices that $$r \geq a = max \{ \frac{s^2 - \epsilon_0^2}{2 (t - \epsilon_0)}, \epsilon_0 \}.$$
\end{proof}

\section{Renz's criteria and the homotopical part of Theorem \ref{ThmI}}  \label{pf2}

Recall that for a group $G$ of type $F_n$ we have $\Sigma^n(G) = \Sigma^2(G) \cap \Sigma^n(G, \Z)$, thus in order to complete the proof of Theorem \ref{ThmI} we only need to consider the homotopical invariant in dimension $2$. The result that we wish to prove can be stated as follows.

\begin{theorem} Let $[H,H] \subseteq K \subseteq H$ be groups such that $H$ and $K$ are finitely presentable. Let $\chi : K \to \R$ be a character such that $\chi([H,H]) = 0$. Then $[\chi] \in \Sigma^2(K)$ if and only if $[\mu] \in \Sigma^2(H)$ for every character $\mu : H \to \R$ that extends $\chi$.
	\end{theorem}

 Note that Theorem \ref{subnormal} is not applicable here, so we need to consider both implications of the statement above. 

\subsection{Extension of characters}
\begin{prop} Let $[H,H] \subseteq K \subseteq H$ be groups such that $H$ and $K$ are finitely presentable. Let $\chi : K \to \R$ be a character such that $\chi([H,H]) = 0$ and $[\chi] \in \Sigma^2(K)$. Then $[\mu] \in \Sigma^2(H)$ for every character $\mu : H \to \R$ that extends $\chi$.
\end{prop}

\begin{proof}
 Notice that $H/K$ is a finitely generated abelian group. By the homotopical version of Theorem \ref{teo-Sigma.de.gr.=.Sigma.de.subgr.} (see \cite[Cor.~2.7]{Me}) we can pass to a finite-index subgroup of $H$ and assume that $H/K$ is free-abelian, so $H$ is built from $K$ by some extensions by $\mathbb{Z}$. By induction we may actually assume that $H = K \rtimes \mathbb{Z}$, and the proposition will follow. 
 
 By \cite[Thm.~B]{Me} if $G$ is an HNN extension with base group $B$ that is finitely presented and associated subgroup $A$ that is finitely presented too and $\mu: G \to \mathbb{R}$ is a character such that $\mu |_A \not= 0$, $[\mu |_A] \in \Sigma^1(A), [\mu|_B] \in \Sigma^2(B)$ then $[\mu] \in \Sigma^2(G)$. We can apply this for $G = H = K \rtimes \mathbb{Z}, B = A = K$ to complete the proof.
 \end{proof}

\subsection{Renz's criteria}
Let $G$ be a finitely generated group and let $X$ be a finite generating set. Suppose that $[\chi] \in S(G)$ is a non-zero character. If $w = x_1 \cdots x_n$ is a word on the generators $X^{\pm 1}$, the \textit{valuation} of $w$ with respect to $\chi$ is 
\[ v_{\chi}(w) = min \{ \chi(x_1 \cdots x_j) \mid 0 \leq j \leq n\}.\]

\begin{theorem} \cite{Renz}  \label{crit1}  Let $G = \langle X \rangle$ with $X$ finite. Then
$[\chi] \in \Sigma^1(G)$ if and only if there is an element $t \in X^{\pm 1}$ such that $\chi(t) > 0$ and such that for each $x \in X^{\pm 1}$ there is a word $w_x$ that represents $t^{-1}xt$ in $G$ and satisfies
\[ v_{\chi}(w_x) > v_{\chi}(t^{-1}xt).\] 
\end{theorem}

Suppose now that $G$ is a finitely presentable group and let $\langle X|R \rangle$ be a finite presentation. 
Let $M$ be a (van Kampen) diagram over $\langle X|R \rangle$. As usual, $M$ comes equipped with a base point $x_0$ (a vertex) on its boundary. In the original Renz definition   any vertex of $M$ is labeled by an element of $G$ as follows: $x_0$ has label $1_G$ and if $x$ is another vertex, then its label is the image $g \in G$ of the label of any edge-path connecting $x_0$ and $x$ inside $M$. 

Now, for any $[\chi] \in S(G)$ and for any diagram $M$, we can define:
\[ v_{\chi}(M) = min \{ \chi(g) \mid \hbox{$g$ labels a vertex of $M$}\}.\]

Suppose that $[\chi] \in \Sigma^1(G)$. Let $t$ and $\{w_x\}_x$ be as in Theorem \ref{crit1}. For a word $r=x_1 \cdots x_n$, with $x_i \in X^{\pm 1}$, we put:
\[ \widehat{r}_t = w_{x_1} \cdots w_{x_n}.\]

\begin{theorem} \cite{Renz}  \label{crit2}
Let $G$ be a finitely presented group and let $G = \langle X | R \rangle$ be a finite presentation. Let $[\chi] \in \Sigma^1(G)$. With $t$ and $\{w_x\}$ as above, suppose further that $R \supseteq \{ t^{-1}xtw_x^{-1} \mid x \in X^{\pm 1} \}$. Then $[\chi] \in \Sigma^2(G)$ if and only if for each $r \in R$, there is a diagram $M$ such that $\partial M = \widehat{r}_t$ and $v_{\chi}(M) + \chi(t) > v_{\chi}(r)$.
\end{theorem}

The boundary $\partial M$ of $M$ above is read from the base point in any direction. What we really need is the version of Theorem \ref{BR-compact} for $\Sigma^2$, which is the following.

\begin{theorem}  \cite{Renzthesis}  \label{thmcpct}
Let $G$ be a finitely presentable group and let $\Gamma \subseteq S(G)$ be a non-empty compact set. Then $\Gamma \subseteq \Sigma^2(G)$ if and only if there exist a finite presentation $G = \langle X | R \rangle$, a finite set $\mathcal{W}$ of words with letters in $X^{\pm 1}$ and a finite set $\mathcal{M}$ of diagrams over $\langle X | R \rangle$ such that for each $[\chi] \in \Gamma$, there exists $t \in X^{\pm 1}$ such that $\chi(t)> 0$ and
\begin{enumerate}
 \item For any $x \in X^{\pm 1}$ there is some $w_{x} \in \mathcal W$ that represents $t^{-1}xt$ in $G$ and $v_{\chi}( w_x)>v_{\chi}(t^{-1}xt)$.
 \item For any $r \in R^{\pm 1}$ there is some $M_{\widehat{r}_t} \in \mathcal{M}$ such that $\partial M_{ {  \widehat{r}_t}}= \widehat{r}_t$ and $v_{\chi}(M_{  \widehat{r}_t}) + \chi(t) > v_{\chi}(r)$, where $\widehat{r}_t$ is computed as above.
\end{enumerate}
\end{theorem}

{\it Remark:} The element $w_x$ actually depends on both $x$ and $\chi$.

\subsection{The converse}
\begin{prop} \label{converse}
Let $[H,H] \subseteq K \subseteq H$ be groups such that $H$ and $K$ are finitely presentable. Let $\chi : K \to \R$ be a character such that $\chi([H,H]) = 0$. Suppose that $[\mu] \in \Sigma^2(H)$ for every character $\mu : H \to \R$ that extends $\chi$. Then $[\chi] \in \Sigma^2(K)$.
\end{prop}

We will use here the same notations as in Section \ref{difficult} except that $t$ here denotes element of $X^{ \pm 1}$ and not some positive real number. In particular, $H_0 \subseteq H$ is a finite-index subgroup such that $[H,H] \subseteq H_0$ and both $H_0 / [H,H]$ and $H_0/K_0$ are torsion-free, where $K_0 = K \cap H_0$. 
We will also keep the notations for $V$, $W$, $D_r$, $\alpha$ and so on. We denote $\chi_0 = \chi |_{K_0}$ and $\mu_0 = \mu |_{H_0}$ for any extension $\mu: H \to \R$ of $\chi$. Again by \cite[Cor.~2.7]{Me} we have that $[\mu _0] \in \Sigma^2(H_0)$ for all extensions $\mu: H \to \R$ of $\chi$ and, once we show that $[\chi_0] \in \Sigma^2(K_0)$, then also $[\chi] \in \Sigma^2(K)$.

Let $H_0 = \langle X |R \rangle$ be a finite presentation of $H_0$ satisfying the conditions of Theorem \ref{thmcpct}
with respect to $$\Gamma = S(H_0,K_0) \cup \{   [\mu_0]  \mid  \mu_0 |_{K_0} = \chi_0\}  \subseteq \Sigma^2(H_0).$$ Thus we are given the finite sets $\mathcal{W}$ and $\mathcal{M}$ of words and diagrams satisfying the required properties (1) and (2) from Theorem \ref{thmcpct}. 

We define now some constants that will appear in the sequence. First, let
\[ a_1= inf \{v_{\varphi}(w_x) - v_{\varphi}(t^{-1} x t)\},\]
where the infimum is taken over all $x \in X^{\pm 1}$, all $\varphi$ such that $[\varphi] \in \Gamma$ and $\| \varphi\|=1$, all $t \in X^{\pm 1}$ such that $\varphi(t)>0$, and all $w_x \in \mathcal{W}$ such that $t^{-1}xt = w_x$ in $H_0$ and $v_{\varphi}(w_x) > v_{\varphi}(t^{-1}xt)$. Similarly, we put
\[ a_2= inf \{ v_{\varphi}({M_{\widehat{r}_t}}) +\varphi (t) - v_{\varphi}(r)\},\]
 where the infimum is taken over all $r \in R^{\pm 1}$, $\varphi \in \Gamma$ with $|| \varphi || = 1$, all $t \in X^{\pm 1}$ such that $\varphi(t)>0$ and all $M_{\widehat{r}_t} \in \mathcal{M}$  such that $v_{\varphi}({M_{\widehat{r}_t}}) +  \varphi(t) > v_{\varphi}(r)$. Notice that  since $\Gamma$ is a compact set and by the conditions of Theorem \ref{thmcpct}  both $a_1$ and $a_2$ are positive real numbers. We define
 \begin{equation}  \label{def.a}
 a= inf\{ a_1, a_2\} .
 \end{equation}
We also define
\[ b_1 = sup \{ \| \pi(tw) \|, \| \pi(x^{-1}tw) \| \},\]
where the supremum runs over all $t,x \in X^{\pm 1}$ and all initial subwords $w$ of all $w'\in \mathcal{W}$, 
and $$\pi : H_0 \to Q_0 = H_0 / [H,H]$$ is the canonical projection. We define
\[ b_2 = sup \{ \| \pi(tg h^{-1}) \| \},\]
where the supremum runs over all $t \in X^{\pm 1}$ and all $g$ and $h$ are vertices of some $M \in \mathcal{M}$, for all possible $M$. Finally, we set
\begin{equation}  \label{def.b}
 b= sup \{b_1, b_2, 2a\}.
\end{equation}

Let $C$ be the Cayley graph of $H_0$ with respect to the generating set $X$. Define $C_d$ as the full subgraph of $C$ spanned by the vertices $h \in H_0$ such that $\alpha(\{h\}) \leq d$,  where $\alpha$ was defined in (\ref{Delta}).

Fix some $0< \epsilon< a$. 

	\medskip
	{\bf Claim } {\it For $d>max\{\frac{b^2-\epsilon^2}{2(a-\epsilon)}, a\}$, the subgraph $C_d$ is connected.}
		\medskip

{\bf Proof of the Claim} Let $h_0$ be a vertex of $C_d$ and let $\gamma$ be an edge-path in $C$ connecting $1=1_{H_0}$ and $h_0$. Let $p$ be the label of $\gamma$ and let $Vert(\gamma)$ be the set of vertices of $\gamma$. If $\alpha(Vert(\gamma)) \leq d$, then $\gamma$ runs inside $C_d$. Otherwise, let $g \in Vert(\gamma) \smallsetminus \{1,h_0\}$ such that
\[ \alpha(Vert(\gamma)) = \alpha (\{g\}) = s > d.\]
As in the proof of Theorem \ref{diffthm}, there must be some $w\in W_{\chi \geq 0}$ such that $\| w-\pi(g) \|=s$ and such that the equivalence class of the character $\nu: H_0 \to \R$ defined by $$\nu(h) = \langle \pi(h), \frac{w-\pi(g)}{ \|w- \pi(g) \|} \rangle$$ lies in $\Gamma$. Notice that $\nu$ attains its minimum on $Vert(\gamma)$ at $g$.

Choose $t \in X^{\pm 1}$ such that $\nu(t)>0$ as in Theorem \ref{thmcpct}. Write $p= x_1  \cdots x_n$ for some $x_1, \ldots, x_n \in X^{ \pm 1}$, so $g = x_1 \cdots x_j$ for some $j <n$. We can find words $w_j, w_{j+1} \in \mathcal{W}$ such that $t^{-1}x_i t = w_i$ and $v_{\nu}(w_i)> v_{\nu}( t^{-1}x_i t)$ for $i=j,j+1$. Thus the path $\gamma'$ starting at $1$ and having label \[p' = x_1 \cdots x_{j-1} t w_j w_{j+1} t^{-1} x_{j+2} \cdots x_n\]
also ends at $h_0$.

Let $y$ be a vertex in $Vert(\gamma') \smallsetminus Vert(\gamma)$. Thus either $y = g t w'$, where $w'$ is an initial subword of $w_{j+1}$, or $y = g x_j^{-1}tw'$, where $w'$ is an initial subword $w_j$. In any case we have
\begin{equation}  \label{rest1}
  \| \pi(y) - \pi(g) \| \leq b,
\end{equation}
where $b$ is defined as in \eqref{def.b}.

On the other hand, we have:
\[ \nu(y)= \nu(g) + \nu(tw'),\]
in case $y = g t w'$ and $w'$ is an initial subword of $w_{j+1}$,  or 
\[ \nu(y)= \nu(g) + \nu(x_j^{-1}tw')\] 
if $y = g x_j^{-1}tw'$, where $w'$ is an initial subword of $w_j$. Using that $\nu$ attains its minimum on $Vert(\gamma)$ at $g$, we have in any case
\begin{equation} \label{rest2}
 \nu(y) \geq \nu(g) + a_1 \geq \nu(g) + a,
 \end{equation}
where $a$ is defined  in \eqref{def.a}.

The argument with Pythagoras' Theorem applied to restrictions \eqref{rest1} and \eqref{rest2} implies that $\|\pi(y)-w\|<s-\epsilon$, so $\alpha(\{y\})< \alpha(\{g\})-\epsilon$. It follows that 
either $\alpha(Vert(\gamma')) < s-\epsilon$, or 
$s - \epsilon \leq \alpha(Vert(\gamma')) \leq s$, but $\gamma'$ has less vertices $g'$ with $\alpha(\{g'\})\geq s-\epsilon$.

All of this implies that $C_d$ is connected for $d>max\{\frac{b^2-\epsilon^2}{2(a-\epsilon)}, a\}$. Indeed, if $h \in C_d$ is a vertex and $\gamma$ is any path connecting $1$ and $h$ in $C$, then it can be transformed by some applications of the procedure described above to another path $\gamma_2$ such that 
\[\alpha(Vert(\gamma_2))\leq max \{\alpha(Vert(\gamma))-\epsilon, d\}.\] Then we iterate this until we find $\gamma_n$ such that $\alpha(Vert(\gamma_n)) \leq d$, so that $1$ and $h$ are connected inside $C_d$. {This proves the claim.

\begin{inote}  \label{Rem1}
 This argument actually shows that the full subgraphs $C_{d,  \chi_0 \geq j}$ spanned by the vertices $h \in C_d$
  such that $\pi(h) \in D_{d, \chi_0 \geq j}$ are connected for any $j$ and for sufficiently large $d$, where  $D_{d,\chi_0 \geq j}$ is defined in (\ref{needed}). By definition $C_d  = C_{d, \chi_0 \geq 0}$ and the proof for $C_d$ works for $C_{d, \chi_0 \geq j}$ with the only difference that the path $\gamma$ starts at a fixed  vertex in $C_{d, \chi_0 \geq j}$ instead of $ 1 = 1_{H_0}$. 
\end{inote}

Now we consider the Cayley complex $\mathcal{C}$ of $H_0$ associated to the presentation $\langle X |R\rangle$ and the full subcomplex $\mathcal{C}_d$, spanned by all $h \in H_0$ such that $\alpha(\{h\}) \leq d$, for $d>max\{\frac{b^2-\epsilon^2}{2(a-\epsilon)}, a\}$. We already know that it is connected since $C_d$ is the $1$-skeleton of $\mathcal{C}_d$. We want to show that $\mathcal{C}_d$ is also simply-connected.

Let $\gamma \subseteq \mathcal{C}_d$ be a closed path with basepoint $1$ and label $p=x_1 \cdots x_n$. Since $\mathcal{C}$ is simply-connected, there is a diagram $D$ over $\langle X|R\rangle$ such that $\partial D = \gamma$. If 
$\alpha(Vert(D)) \leq d$, we are done. Otherwise there is an internal vertex $g$ such that 
\[ \alpha(Vert(D)) = \alpha (\{g\}) = s > d.\]
We consider again the character $\nu: H_0 \to \R$, defined by $\nu(h) = \langle \pi(h),\frac{ w- \pi(g)}{ \|w - \pi(g)\|} \rangle$, where $w\in W_{\chi_0 \geq 0}$ and $\|w-\pi(g)\| = s$. Recall that $[\nu] \in \Gamma$.

Choose $t \in X^{\pm 1}$ such that $\nu(t)>0$. Then for each $r\in \mathcal{R}^{\pm 1}$ we can find a diagram $M_{\hat{r}_t} \in \mathcal{M}$ satisfying the conditions of the $\Sigma^2$-criterion,  i.e. (2) from Theorem \ref{thmcpct}. This, together with the cells associated to $t^{-1}xtw_x^{-1}$ for the letters $x$ of $r$, gives a new diagram $M_r'$ such that $\partial M_r'= r$ and such that 
$\nu(h) >v_{\nu}(r)$ for all vertices $h$ of $M_r'$ that do not lie on the boundary of the diagram. 

We transform $D$ as follows: we substitute the cells $e_r$, associated to relators $r$, having $g$ as a vertex, with the diagrams $M_r'$. By canceling the pairs of adjacent cells $e_{\rho}$ and $e_{\rho^{-1}}$ for some relators $\rho$ of type $t^{-1}xtw_x^{-1}$, we obtain a new diagram $D'$ that has the same boundary but does not contain $g$. This is the same argument as in the proof of the $\Sigma^2$-criterion (Theorem \ref{crit2}) in \cite{Renz}. Any vertex $y$ of $D'$ that was not a vertex of $D$ can be written as $y=gtw'$, where $w'$ is the image in $G$ of the label of a path from a vertex $v_0 = gt$ from the boundary of some $M_{\widehat{r}_t}$ to some other vertex  of $M_{\widehat{r}_t}$.

By the definition of $b_2$ we have
\begin{equation}  \label{rest3}
  || \pi(y) - \pi(g) || \leq b_2 \leq  \  b.
\end{equation}

Recall that $r$ is a relation that defines a 2-cell $e_r$, i.e. the closed path $\partial e_r$ has label $r$ and that $g$ is a vertex of $e_r$. Without loss of generality we can assume that $g$ is the base point of this cell (a similar argument works when the base point is different from $g$). Thus if $r = x_1 \ldots x_m$ for $x_1, \ldots, x_m \in X^{  \pm 1}$ the consecutive vertices of $\partial( e_r)$ are $g, g \overline{x}_1, \ldots, g \overline{x}_1 \ldots \overline{x}_i, \ldots, g \overline{x}_1 \ldots \overline{x}_{m-1}$, where $\overline{x}_i$ is the image of $x_i$ in $G$. Since $\nu(g)$ is minimal among $\{\nu(v) ~| v \in Vert(D) \}$ we conclude that $v_{\nu}(r) = 0$.

Recall that $M_{\widehat{r}_t}$ is attached at the vertex $gt$, that is actually the base point (as $t$ is the base point of $\partial ( e_r) = r$). In the original Renz definition a base point  is labeled by $1_G$ but in the Cayley complex it is labeled by $gt$, hence we obtain that
$$v_{\nu}(M_{\widehat{r}_t}) = min \{ \nu(h) - \nu(gt) ~ | ~ h \in Vert( M_{\widehat{r}_t} ) \} \leq \nu(gtw') - \nu(gt).$$
Thus  we have
\begin{equation} \label{rest4}
 \nu(y) = \nu(gtw') \geq  \nu(gt) + v_{\nu}(M_{\widehat{r}_t}) \geq \nu(gt) + v_{\nu}(r) - \nu(t) + a_2 =  \end{equation} $$  \nu(gt) - \nu(t) + a_2 \geq  \nu(gt) - \nu(t) + a = \nu(g) + a.
$$

 Again by Pythagoras' Theorem we find that the new vertices $y$ of $D'$ satisfy $\alpha(\{y\}) <\alpha(\{g\})-\epsilon$. This procedure defines the transformation $D \mapsto D_2$ such that 
 \[ \alpha(Vert(D_2)) \leq max \{\alpha(Vert(D))-\epsilon,d\}.\]
 We can iterate this to produce a diagram $D_n$ with $\partial D_n = \partial D$ and $\alpha(Vert(D_n)) \leq d$. 
  This completes the proof of the fact that $\mathcal{C}_{d}$ is simply connected for sufficiently big $d$.

 \begin{inote}
  Again, we can use the argument to show that the full subcomplex $\mathcal{C}_{d, \chi \geq j}$ spanned by the vertices of $C_{d,\chi \geq j}$ is simply-connected for any $j$ and for $d$ big enough.
 \end{inote}

 {\bf Proof of Proposition \ref{converse}} Fix $d$ such that $\mathcal{C}_d$ is $1$-connected. Let $\widetilde{\mathcal{C}}_d$ be the full subcomplex of $\mathcal{C}$ spanned by the vertices $h$ such that $\pi_0(h) \in \widehat{B}_d$, where $\pi_0: H_0 \to H_0/K_0 \simeq \mathbb{Z}^{k-m}$ is the canonical projection and as defined  before in the proof of Theorem \ref{diffthm}  $\widehat{B}_d$ is the closed ball with center $0$ and radius $d$ in $H_0/K_0 \otimes \mathbb{R}$. Notice that $\widetilde{\mathcal{C}}_d$ is a free $K_0$-complex with finitely many orbits of cells, and it is also $1$-connected by the original argument in \cite{Renzthesis}.
 
Let $\pi_1: H_0 \to K_0/[H,H]$ be the projection defined in (\ref{pi-1}). Let $\beta:H_0 \to \R$ be the composite of $\pi_1$ with 
the induced character $\bar{\chi_0}:  K_0/[H,H] \to \R$. Then $\beta$ extends to a $\chi_0$-equivariant regular height function (in the sense of  \cite[Section~2]{Me}) defined on $\widetilde{\mathcal{C}}_d$. We can then consider the valuation subcomplexes $(\widetilde{\mathcal{C}}_d)_{\beta \geq j}$: this is simply the full subcomplex of $\widetilde{\mathcal{C}}_d$ spanned by the vertices $h$ such that $\beta(h) \geq j$.

Finally, we can find some fixed negative numbers $j_1>j_2$, depending only on $d$, such that 
\[(\widetilde{\mathcal{C}}_d)_{\beta \geq j} \subseteq \mathcal{C}_{d, \chi \geq j+j_1} \subseteq (\widetilde{\mathcal{C}}_d)_{\beta\geq j+j_2}\]
for all $j$. Since $\mathcal{C}_{d, \chi \geq j}$ is $1$-connected for all $j$, the induced maps 
\[\pi_i ((\widetilde{\mathcal{C}}_d)_{\beta \geq j}) \to \pi_i((\widetilde{\mathcal{C}}_d)_{\beta \geq j+j_2})\] are trivial for $i=0,1$. Thus $[\chi_0] \in \Sigma^2(K_0)$ by \cite[Theorem A]{Me}.

\end{document}